\documentclass[reqno]{amsart}

\usepackage[dvipdfmx]{color}

\usepackage{amsmath,amssymb,mathrsfs}

\makeatletter
\def\@settitle{\begin{center}\baselineskip14\p@\relax\bfseries{\large\@title}\thispagestyle{empty}\end{center}}
\def\@setauthors{%
  \begingroup
  \def\thanks{\protect\thanks@warning}%
  \trivlist
  \centering\footnotesize \@topsep30\p@\relax
  \advance\@topsep by -\baselineskip
  \item\relax
  \author@andify\authors
  \def\\{\protect\linebreak}%
  {\authors}%
  \ifx\@empty\contribs
  \else
    ,\penalty-3 \space \@setcontribs
    \@closetoccontribs
  \fi
  \endtrivlist
  \endgroup
}
\def\maketitle{\par
  \@topnum\z@ 
  \@setcopyright
  \thispagestyle{firstpage}
  \ifx\@empty\shortauthors \let\shortauthors\shorttitle
  \else \andify\shortauthors
  \fi
  \@maketitle@hook
  \begingroup
  \@maketitle
  \toks@\@xp{\shortauthors}\@temptokena\@xp{\shorttitle}%
  \toks4{\def\\{ \ignorespaces}}
  \edef\@tempa{%
    \@nx\markboth{\the\toks4
      \@nx
      {\the\toks@}}{\the\@temptokena}}%
  \@tempa
  \endgroup
  \c@footnote\z@
  \@cleartopmattertags
}
\makeatother

\newtheorem{proposition}{Proposition}[section]
\newtheorem{theorem}[proposition]{Theorem}
\newtheorem{lemma}[proposition]{Lemma}
\newtheorem{corollary}[proposition]{Corollary}
\theoremstyle{definition}

\newtheorem{example}[proposition]{Example}
\newtheorem{remark}[proposition]{Remark}

\newcommand{\Q}{\mathbb Q}
\newcommand{\R}{\mathbb R}
\newcommand{\C}{\mathbb C}
\newcommand{\ce}{c_{\ue}}
\newcommand{\Digit}[1][r]{\mathcal{A}_{#1}}

\newcommand{\Fourier}[1][f]{\mathcal{F}[#1]}
\newcommand{\rapid}[1][V]{\mathscr{S}(#1)}
\newcommand{\lzf}[3][\ue]{\Phi_{#1}(#2;\,#3)}
\newcommand{\dlzf}[3][\ud]{\Phi^*_{#1}(#2;\,#3)}
\newcommand{\vlzf}[2]{\bs{\Phi}(#1;\,#2)}
\newcommand{\vdlzf}[2]{\bs{\Phi}^*(#1;\,#2)}
\newcommand{\clzf}[2]{\bs{\Psi}(#1;\,#2)}
\newcommand{\cdlzf}[2]{\bs{\Psi}^*(#1;\,#2)}
\newcommand{\sign}[1][\ul{a}]{(\sqrt{-1})^{|#1|}}

\newcommand{\pmat}[1]{\begin{pmatrix} #1 \end{pmatrix}}
\newcommand{\smat}[1]{\left(\begin{smallmatrix} #1 \end{smallmatrix}\right)}
\newcommand{\set}[2]{\left\{#1 ;\; #2 \right\}}

\newcommand{\s}{\ul{s}}
\newcommand{\ds}{\displaystyle}
\newcommand{\efactor}{\mathcal{E}}
\newcommand{\w}{\ul{w}}

\newcommand{\kvsup}[1]{\kappa_{\ue}(#1)}
\newcommand{\localZ}[2][\ul{a}]{Z_{#1}(#2;\,\ul{s})}
\newcommand{\dlocalZ}[2][\ul{b}]{Z^*_{#1}(#2;\,\ul{s})}
\newcommand{\revProd}[1][j=1]{\prod_{#1}^r\rule{0pt}{10pt}^{\rm rev}}
\newcommand{\bstrut}{\rule[-2pt]{0pt}{12pt}}

\newcommand{\ctt}{\mathtt{c}}
\newcommand{\stt}{\mathtt{s}}
\newcommand{\ul}[1]{\underline{#1}}
\newcommand{\ud}{\boldsymbol{\delta}}
\newcommand{\ue}{\boldsymbol{\varepsilon}}

\newcommand{\bs}[1]{\boldsymbol{#1}}

\newcommand{\Ir}[1][r]{\mathcal{I}_{#1}}
\newcommand{\dIr}{\mathcal{I}^0_r}
\newcommand{\Oe}[1][\ue]{\mathcal{O}_{#1}}
\newcommand{\dOe}{\mathcal{O}^*_{\ud}}

\newcommand{\kv}[1][\ul{a}]{\bs{\kappa}(#1)}
\newcommand{\kvsub}[2][\ul{a}]{\kappa_{#2}(#1)}
\newcommand{\pcos}[2][\ul{a}]{Q_{#1}(#2)}
\newcommand{\Pcos}[2][\ul{a}]{Q_{#1}{#2}}

\newcommand{\uluelement}[1][j]{w_j}
\newcommand{\innV}[2]
	{\left\langle \mskip2.0 mu #1 \mskip2.0 mu | 
	\mskip2.0 mu #2 \mskip2.0 mu \right\rangle}

\newcommand{\bl}[1]{\textcolor{blue}{#1}}

\title
{
On gamma matrices of local zeta functions associated with homogeneous cones}

\author[H.\ Nakashima]
{Hideto Nakashima}
\address{
The Institute of Statistical Mathematics\\ 
Midori-cho 10-3, Tachikawa, Tokyo 190-8562, 
Japan
}
\email{hideto@ism.ac.jp}


\begin{document}

\begin{abstract}
The purpose of this paper is 
to investigate coefficient matrices of functional equations of 
zeta functions associated with homogeneous cones,
which are given explicitly in the previous paper,
in detail.
We prove that the coefficient matrix can be decomposed into variable-wise matrices
regardless of the choice of homogeneous cones.
Moreover,
under a suitable condition,
we show that
the associated zeta functions admit a kind of completion forms.
\end{abstract}

\maketitle

\section{Introduction}
A theory of prehomogeneous vector spaces, constructed by Sato~\cite{MSato} (see also Sato--Shintani~\cite{MSatoShintani}, Kimura~\cite[Introduction]{Kimura}),
provides a systematic method to construct zeta functions satisfying functional equations.
Gamma matrices, which are the coefficient matrices of functional equations, 
possess important information of zeta functions,
and they have been determined for many class of prehomogeneous vector spaces
(cf.\ Shintani~\cite{Shintani}, Muro~\cite{Muro1980, Muro1986}, Sato and Shintani~\cite{MSatoShintani}).
In the case of the Riemann zeta function,
we are able to decompose the gamma matrix into two factors to
obtain a completion of functional equation.
This means that
$\zeta(s)$ can be completed to the Riemann xi function 
$\xi(s):=
2^{-1}s(s-1)
\pi^{-s/2}\Gamma(s/2)\zeta(s)$
whose functional equation has a higher symmetric form
$\xi(1-s)=\xi(s)$.
Here, $\Gamma(s)$ is the ordinary gamma function.
It is confirmed that a similar phenomenon occurs for zeta functions associated with prehomogeneous vector spaces of binary cubic form 
(cf. Ohno~\cite{Ohno} and Nakagawa~\cite{Nakagawa}, see also Datsovsky--Wright~\cite{DatsovskyWright}),
but 
it is not known whether gamma matrices of other prehomogeneous vector spaces 
admit a kind of completion or diagonalization,
or weaker form that
they can be decomposed into a product of variable-wise matrices.

Let us explain these phenomena
by a concrete example.
Let $(G,\rho,V)$ be a prehomogeneous vector space in Sato~\cite[\S\S7.1 Example (A)]{FSatoI}.
We use all notation in that paper without comments, and assume $v(L^{(1)*})=1$ for simplicity.
For dual zeta functions, we choose $E$ as a $\Q$-regular subspace.
Put $\ctt(z)=\cos\frac{\pi z}{2}$, $\stt(z)=\sin\frac{\pi z}{2}$ for $z\in\C$,
and
$\tau(\ul{s}):=(s_1+s_2+s_3-1,\ 1-s_3,\ 1-s_2)$ for $\ul{s}=(s_1,s_2,s_3)\in\C^3$.
Then, we have 
\[
\pmat{\xi_+(L^*_E;\,\tau(\ul{s}))\\ \xi_-(L^*_E;\,\tau(\ul{s}))}
=
\frac{2\Gamma(s_2)\Gamma(s_3)}{(2\pi)^{s_2+s_3}}
\pmat{
\ctt(s_2+s_3)
&
\ctt(s_2-s_3)\\
\ctt(s_2-s_3)
&
\ctt(s_2+s_3)}
\pmat{\xi_+(L;\ul{s})\\ \xi_-(L;\,\ul{s})}
\]
by Theorem~3 (iii) of that paper~\cite{FSatoI}.
This gamma matrix can be obviously diagonalized, 
and moreover using Euler's reflection formula and Legendre's duplication formula 
of the gamma function, 
we obtain a functional equation with a higher symmetric form as follows.
Set
\[
B(\ul{s})=\pi^{-\tfrac{s_2+s_3}{2}}
\pmat{\Gamma\bigl(\frac{s_2}{2}\bigr)\Gamma\bigl(\frac{s_3}{2}\bigr)&0\\
0&\Gamma\bigl(\frac{s_2+1}{2}\bigr)\Gamma\bigl(\frac{s_3+1}{2}\bigr)}
\cdot\frac{1}{\sqrt{2}}\pmat{1&1\\1&-1}.
\]
Then, we have
\[
\bs{\eta}^*_E\bigl(\tau(\ul{s})\bigr)
=
\pmat{1&0\\0&-1}
\bs{\eta}_E(\ul{s}),\]
where
\[\bs{\eta}^*_E(\ul{s})=B(\ul{s})\pmat{\xi_+(L^*_E;\,\ul{s})\\ \xi_-(L^*_E;\,\ul{s})},\quad
\bs{\eta}_E(\ul{s})
=
B(\ul{s})\pmat{\xi_+(L;\ul{s})\\ \xi_-(L;\,\ul{s})}.
\]

In this paper,
we shall say that vector-valued zeta functions 
$\bs{\zeta}(\ul{s})$ and $\bs{\zeta}^*(\ul{s})$
associated with regular prehomogeneous vector spaces 
can be completed if
there exist square matrices $B(\ul{s})$ and $B^*(\ul{s})$ such that
\[\bs{\xi}^*(\tau(\s))=\efactor\bs{\xi}(\s)\quad\text{where}\quad
\bs{\xi}(\s):=B(\s)\bs{\zeta}(\s)\text{ and }\bs{\xi}^*(\s):=B^*(\s)\bs{\zeta}^*(\s),\]
where $\efactor$ is the so-called $\varepsilon$-factor,
that is, a diagonal matrix with entries $\exp(\theta\sqrt{-1})$ $(\theta\in\R)$.
Since $B(\ul{s})$ and $B^*(\ul{s})$ may be different matrices,
there is a trivial solution $B(\s)=A(\s)$, $B^*(\s)=I$ and $\efactor=I$ ($I$ is the identity matrix)
and hence we would like to make $B(\s)$ and $B^*(\s)$ as similar as possible.
Then, it can be confirmed that
the other dual zeta functions in \S\S7.1 of \cite{FSatoI} can be also completed
in this sense.

On the other hand,
for another example \cite[\S\S7.2 Example (B)]{FSatoI} (cf.\ Shintani~\cite[Chapter 1]{Shintani}), 
the following functional equation in \cite[Theorem~4 (ii)]{FSatoI} seems not to be diagonalized, and hence not to be completed.
\[
\begin{array}{r@{\ }c@{\ }l}
\ds
\pmat{\xi_1(s_1,\frac{3}{2}-s_1-s_2)\\
\xi_2(s_1,\frac32-s_1-s_2)}
&=&
\ds
\frac{\sqrt{\pi}}{2}
\Bigl(\frac{2}{\pi}\Bigr)^{s_1+2s_2}
\Gamma(s_2)\Gamma\Bigl(s_1+s_2-\frac12\Bigr)\\
& &
\quad
\times
\pmat{\stt(s_1+2s_2)&\stt(s_1)\\ \ctt(s_1)&\ctt(s_1+2s_2)}
\pmat{\xi^*_1(s_1,s_2)\\ \xi^*_2(s_1,s_2)}.
\end{array}
\]
However,
if we set 
\[
\bs{\eta}(\ul{s})=\frac{1}{\sqrt{2}}\pmat{1&1\\ 1&-1}
\pmat{\xi_1(s_1,s_2)\\
\xi_2(s_1,s_2)},\quad
\bs{\eta}^*(\ul{s})=\frac{1}{\sqrt{2}}\pmat{1&1\\ 1&-1}
\pmat{\xi_1^*(s_1,s_2)\\ \xi_2^*(s_1,s_2)},
\]
then the gamma matrix can be decomposed into
a product of variable-wise orthogonal matrix and diagonal matrices in a simple form as
\[
\bs{\eta}(\tau(\ul{s}))
=
\frac{2^\alpha\Gamma(\alpha)}{\pi^\alpha}
\pmat{\ctt(\alpha)&0\\0&-\stt(\alpha) }
\cdot\pmat{1&1\\-1&1}
\cdot
\frac{2^{s_2}\Gamma(s_2)}{\pi^{s_2}}
\pmat{\ctt(s_2)&0\\0&\stt(s_2)}
\bs{\eta}^*(\ul{s}),
\]
where $\tau(\ul{s})=(s_1,\frac32-s_1-s_2)$ and $\alpha=s_1+s_2-\frac12$.
These observations seem to be new.
In general,
except for some particular reductive cases
(cf.\ \cite{SatakeFaraut}, \cite{DatsovskyWright} and \cite{Thor}),
it is not known whether gamma matrices associated with
prehomogeneous zeta functions can be decomposed into matrices in a simple form or not,
or completed or not.

In this paper,
we consider this problem for prehomogeneous vector spaces
associated with homogeneous open convex cones containing no entire line
(homogeneous cones for short in what follows)
studied in the previous paper~\cite{N2020},
and we show that the associated gamma matrices can always be decomposed into
variable-wise matrices.
Moreover,
we see that the associated zeta functions have completions
for a certain class of homogeneous cones.
We actually work with local zeta functions (see \eqref{eq:lzf} for definition)
instead of their zeta functions for some reasons.
One reason is that 
the gamma matrix of zeta functions can be described by using 
that of local zeta functions
(cf.\ \eqref{eq:FEzeta}, see also \cite{N2020}),
and the other is that
we can understand from where the conjugate matrix comes
(see Lemma~\ref{lemma:transform}).
It is worthy to mention that our prehomogeneous vector spaces
are not reductive but solvable.

We now describe the contents of this paper in more detail.
Let $\Omega$ be a homogeneous cone of rank $r$ in an $n$-dimensional real vector space $V$.
Then,
there exists a split solvable Lie group $H$ acting on $\Omega$ linearly and simply transitively
(cf.\ Vinberg \cite{Vinberg}),
and the pair $(H,V)$ admits a structure of real prehomogeneous vector spaces
(cf.\ Gindikin~\cite{Gindikin64}).
We denote the associated basic relative invariants by
$\Delta_1(x),\dots,\Delta_r(x)$ $(x\in V)$ 
(cf.\ Ishi~\cite{Ishi2001}).
The space $V$ admits a weight space decomposition $V=\bigoplus_{1\le j\le k\le r}V_{kj}$ with $\dim V_{jj}=1$ $(j=1,\dots,r)$.
Set $n_{kj}=\dim V_{kj}$ $(1\le j\le k\le r)$.
Let $\ul{p}$, $\ul{q}$ and $\ul{d}$ be the vectors described by using $n_{kj}$
(see \eqref{def:structure constants} for definition).
Open $H$-orbits in $V$ can be parametrized by $\ue\in\Ir:=\{1,-1\}^r$,
and to each open orbit,
we associate a local zeta function $\lzf{f}{\ul{s}}$
(see \eqref{eq:lzf} for definition).
We identify the dual vector space $V^*$ of $V$ with $V$ through
a suitable fixed inner product $\innV{\cdot}{\cdot}$ of $V$.
Then, the dual cone $\Omega^*$ of $\Omega$, defined by using this inner product,
is also a homogeneous cone on which $H$ acts linearly and simply transitively.
We denote by $\sigma$ and $\sigma_*$ the multiplier matrices of $\Omega$ and $\Omega^*$,
respectively (see \eqref{eq:multiplier matrix} for definition).

Set $\Digit=\{0,1\}^r$.
We introduce a suitable order in $\Digit$ in Section~\ref{sect:preliminary}.
Let $\mathrm{sgn}$ be the sign function of $\R^{\times}=\R\setminus\{0\}$,
that is,
$\mathrm{sgn}(x)=x/|x|$, and we set $\mathrm{sgn}(0):=0$.
Let $\omega^{s,a}$ $(s\in\C,\ a=0,1)$ be the quasi-character of $\R^{\times}$
defined by $\omega^{s,a}(x):=\mathrm{sgn}(x)^a|x|^s$.
Using this, we introduce zeta distributions $\localZ{f}$ by
\[
\localZ{f}
:=
\int_V \prod_{i=1}^r\omega^{s_i,a_i}\bigl(\Delta_i(x)\bigr)\,f(x)d\mu(x)
\quad
(\ul{a}\in\Digit,\ f\in\rapid,\ \ul{s}\in\C^r),
\]
where $d\mu$ is a suitable $H$-invariant measure on $V$.
We see in Lemma~\ref{lemma:transform} that
$\localZ{f}$ can be described as a linear combination
of local zeta functions $\lzf{f}{\ul{s}}$.
Similarly,
we can introduce zeta distributions $\dlocalZ{f^*}$ $(\ul{b}\in\Digit,\ \ul{s}\in\C^r)$ for $\Omega^*$.

The Fourier transform $\Fourier$ of $f\in\rapid$ is defined as
\[
\Fourier(x) := \int_Vf(y)\exp\bigl(2\pi\sqrt{-1}\innV{x}{y}\bigr)\,dy
\]
where $dy$ is the Euclidean measure on $V$.
Since we have a functional equation with respect to local zeta functions $\lzf{f}{\ul{s}}$ in \cite[Proposition~4.3]{N2020},
zeta distributions $\localZ{f}$ and $\dlocalZ{f^*}$ also admit a functional equation.

The main theorem of this paper, Theorem~\ref{theorem}, shows that
there exist permutation matrices $W_\sigma$ and $W_{\sigma_*}$,
determined by $\sigma$ and $\sigma_*$ respectively,
such that
\[
W_\sigma\cdot\bigl(\localZ{\mathcal{F}\relax[f^*]}\bigr)_{\ul{a}\in\Digit}
=
\mathcal{C}_r\Bigl(\ul{s}\sigma-\frac{1}{2}\ul{p}\Bigr)
W_{\sigma_*}\cdot\bigl(Z^*_{\ul{b}}(f^*;\,\tau(\ul{s}))\bigr)_{\ul{b}\in\Digit},
\]
where $\tau(\ul{s})=(\ul{d}-\ul{s}\sigma)\sigma_*^{-1}$ and
$\mathcal{C}_r(\ul{\alpha})$ $(\ul{\alpha}\in\C^r)$
is a product of variable-wise matrices as follows.
\[
\mathcal{C}_r(\ul{\alpha})
=
2^r
\pmat{I_{2^{r-1}}&0\\0&\sqrt{-1}I_{2^{r-1}}}
\revProd[j=1]\left(
	\prod_{k=j+1}^r\widetilde{P}^{(r)}_{kj}\Bigl(\frac{n_{kj}}{2}\Bigr)
	\cdot
	\widetilde{D}^{(r)}_j(\alpha_j)
\right).
\]
Here,
$\widetilde{D}^{(r)}_j(\alpha)$ are diagonal matrices
and $P^{(r)}_{kj}(\theta)$ are orthogonal matrices
which are defined inductively 
(see Theorem~\ref{theorem} for details).
Note that, for any fixed $j$ and $r$, the matrices $\widetilde{P}_{kj}^{(r)}(\theta)$ $(k>j)$ are commutative.
Moreover,
for non-commutative variables $a_1,\dots,a_r$, we set
$
(\prod\rule{0pt}{8pt}^{\rm rev})\rule{0pt}{8pt}_{i=1\mathstrut}^r a_i:=a_r\times \cdots a_2\times a_1.
$

In order to consider a completion of a functional equation of local zeta functions,
we suppose that the structure constants $n_{kj}$ of $\Omega$ satisfy a condition that
\[
\text{for fixed $m=0,1$, one has }\frac{\pi}{4}\sum_{j<k}\varepsilon_j\delta_kn_{kj}\equiv m\pi\ (\text{mod }2\pi)
\text{ for all }\ue,\ud\in\Ir.
\] 
Introducing an Hadamard matrix $J^{(r)}$ and a diagonal matrix $\Lambda(\ul{\alpha})$ by
\[
J^{(1)}:=\frac{1}{\sqrt{2}}\pmat{1&1\\1&-1},\quad
J^{(j)}:=J^{(j-1)}\otimes J^{(1)},\quad
\Lambda(\ul{\alpha}):=\mathrm{diag}\left(\frac{\pi^{|\ul{\alpha}|/2}}{\Gamma\bigl(2^{-1}(\ul{\alpha}+\ul{a})\bigr)}\right)_{\ul{a}\in\Digit},\]
we set
\[
\begin{array}{l}
\ds
\bs{\Psi}(f;\,\ul{s}):=\Lambda(\ul{s}\sigma-2^{-1}\ul{p})J^{(r)}\cdot\bigl(\lzf{f}{\ul{s}}\bigr)_{\ue\in\Ir},\\
\ds
\bs{\Psi}^*(f^*;\,\ul{s}):=\Lambda(\ul{s}\sigma_*-2^{-1}\ul{q})J^{(r)}\cdot \bigl(\dlzf{f^*}{\ul{s}}\bigr)_{\ud\in\Ir}.
\end{array}
\]
Then, Proposition~\ref{prop:completion} shows that for $\ul{s}\in\C^r$ and $f^*\in\rapid$
\[
\bs{\Psi}(\mathcal{F}\relax[f^*];\,\ul{s})=\efactor \bs{\Psi}^*(f^*;\,\tau(\ul{s})),\quad
\efactor=(-1)^m\,\mathrm{diag}\Bigl((\sqrt{-1})^{|\ul{a}|}\Bigr)_{\ul{a}\in\Digit}.
\]

We organize this paper as follows.
Section~\ref{sect:preliminary} collects the fundamental facts related to homgoeneous cones 
which we need later.
In particular,
we introduce local zeta functions and zeta distributions associated with homogeneous cones.
In Section~\ref{sect:proof_main},
we shall prove that any gamma matrix corresponding to homogeneous cones
can be decomposed into variable-wise diagonal or orthogonal matrices.
We consider a completion of local zeta functions in Section~\ref{sect:completion}.

\section{Preliminaries}
\label{sect:preliminary}

Let $\Omega$ be a homogeneous cone of rank $r$ in an $n$-dimensional vector space $V$.
We summarize basic properties of homogeneous cones which we need in this paper.
There exists a split solvable Lie group $H$ acting on $\Omega$ linearly and simply transitively
(cf.\ Vinberg~\cite{Vinberg}),
and it is known that
the pair $(H,V)$ has a structure of a real prehomogeneous vector space
(cf.\ Gindikin~\cite[p.\ 77]{Gindikin64}).
The associated basic relative invariants are denoted by $\Delta_1(x),\dots,\Delta_r(x)$.
Note that
$\Omega$ can be described as $\Omega=\set{x\in V}{\Delta_1(x)>0,\dots,\Delta_r(x)>0}$.
Open orbits of $H$ in $V$ are parametrized by $\Ir=\{1,-1\}^r$,
and denote an orbit corresponding to $\ue\in\Ir$ by $\Oe$.
Moreover,
$V$ admits a kind of weight space decomposition
\[
V=\bigoplus_{1\le j\le k\le r}V_{kj},\quad
\dim V_{jj}=1\quad(j=1,\dots,r).
\]
We note that $V_{kj}=\{0\}$ may occur for some $j<k$.
The following integers and vectors are used frequently in this paper without any comments:
\begin{equation}
\label{def:structure constants}
\begin{array}{c}
\displaystyle
n_{kj}:=\dim V_{kj},\quad
p_k:=\sum_{j<k}n_{kj},\quad
q_j:=\sum_{k>j}n_{kj},\\[1em]
\displaystyle
\ul{1}=(1,\dots,1),\quad
\ul{p}=(p_1,\dots,p_r),\quad
\ul{q}=(q_1,\dots,q_r),\quad
\ul{d}=\ul{1}+\frac{1}{2}(\ul{p}+\ul{q}).
\end{array}
\end{equation}
If we take a suitable basis $\{c_j\}$ of a dimensional subspace $V_{jj}$ of $V$ ($j=1,\dots,r$), 
then there exist suitable integers $\sigma_{jk}$ $(1\le j,k\le r)$ such that
\begin{equation}
\label{eq:multiplier matrix}
\Delta_j(x_1c_1+\cdots+x_rc_r)
=
x_{1}^{\sigma_{j1}}\cdots x_r^{\sigma_{jr}}\quad(x_1,\dots,x_r\in\R^\times)
\end{equation}
for $j=1,\dots,r$.
A matrix $\sigma=(\sigma_{jk})_{1\le j,k\le r}$,
which contains an information about multipliers of them,
is called the multiplier matrix of $\Omega$
(cf.\ \cite{N2014}).
Note that $\sigma$ is a unimodular matrix,
that is, all entries are non-negative integers and $\det \sigma = 1$.


We denote by $\rapid$ the Schwartz space of rapidly decreasing functions on $V$.
Let $d\mu(x)$ be a suitable invariant measure on $V$.
For $f\in\rapid$, we put
\begin{equation}
\label{eq:lzf}
\lzf{f}{\ul{s}}:=\int_{\Oe}|\Delta_1(x)|^{s_1}\cdots|\Delta_r(x)|^{s_r}f(x)\,d\mu(x)\quad(\ue\in\Ir,\ \ul{s}\in\C^r)
\end{equation}
which is called the local zeta function associated with $\Oe$.
It is known that $\lzf{f}{\ul{s}}$
are absolutely convergent for $\mathrm{Re}\,\ul{s}>\ul{d}\sigma^{-1}$,
and analytically continued to meromorphic functions of $\ul{s}$ in the whole space $\C^r$ (cf.\ \cite{BG69}).
Note that we write $\ul{\bstrut\alpha}>\ul{\bstrut\beta}$ for $\ul{\bstrut\alpha},\,\ul{\bstrut\beta}\in\R^r$ if $\alpha_j>\beta_j$ for all $j=1,\dots,r$.

Let us take and fix a suitable inner product $\innV{\cdot}{\cdot}$ in $V$.
The dual vector space $V^*$ of $V$ is identified with $V$ through this inner product.
Let $\Omega^*$ be the dual cone of $\Omega$, that is,
$\Omega^*=\set{y\in V}{\innV{x}{y}>0\text{ for all }x\in\overline{\Omega}\setminus\{0\}}$.
Then, $H$ acts also on $\Omega^*$ linearly and simply transitively
via the contragredient representation so that
$\Omega^*$ is also a homogeneous cone.
With respect to $\Omega^*$,
we also have $H$-orbits $\dOe$ parametrized by $\ud\in\Ir$,
the basic relative invariants $\Delta_1^*(y),\dots,\Delta_r^*(y)$ $(y\in V)$,
the multiplier matrix $\sigma_*$,
and local zeta functions $\dlzf{f^*}{\ul{s}}$ $(f^*\in\rapid,\ \ul{s}\in\C^r)$.

Here, we assume that 
orders of the basic relative invariants of $\Omega$ and $\Omega^*$
are determined by the procedure of Ishi~\cite{Ishi2001},
as in the previous paper~\cite{N2020}.
Note that, 
in the case of symmetric cones,
this order of $\Omega^*$ is opposite of the standard one
(cf.\ Faraut and Kor{\'a}nyi~\cite{FK94}).
This makes notation simpler.


Let us equip $\Ir$ with an ordering inductively as follows.
If $r=1$, then we set $\ue_1=1$ and $\ue_2=-1$.
Suppose that the ordering in $\mathcal{I}_{r-1}$ is determined.
Let $\ue_i$ be the $i$-th element of $\mathcal{I}_{r-1}$.
Then,
$(1,\ue_i)$ is the $i$-th element of $\Ir$ and
$(-1,-\ue_i)$ is the ($2^{r-1}+i$)-th element of $\Ir$.
Using this order,
we shall write local zeta functions in a vector form as
\[\vlzf{f}{\ul{s}}=\bigl(\lzf{f}{\ul{s}}\bigr)_{\ue\in\Ir}\quad
\text{and}\quad
\vdlzf{f^*}{\ul{s}}=\bigl(\dlzf{f^*}{\ul{s}}\bigr)_{\ud\in\Ir}.\]

The Fourier transform $\Fourier$ of $f\in\rapid$ is defined as
\[
\Fourier(x):=\int_Vf(y)\exp(2\pi\sqrt{-1}\innV{x}{y})\,dy,
\]
where $dy$ is the Euclidean measure on $V$.
Let $\tau$ be the affine transformation on $\C^r$ defined by
$\tau(\s):=(\ul{d}-\s\sigma)\sigma_*^{-1}$,
where $\sigma_*$ is the multiplier matrix of $\Omega^*$.
For $\ul{\alpha}\in\C^r$,
we write 
$\Gamma(\ul{\alpha}):=\Gamma(\alpha_1)\cdots\Gamma(\alpha_r)$.
Then, the Gindikin gamma function $\Gamma_\Omega(\ul{\alpha})$
of $\Omega$ is defined as
\begin{equation}
\label{eq:defofgamma}
\Gamma_{\Omega}(\ul{\alpha})
=
(2\pi)^{(n-r)/2}\Gamma\Bigl(\ul{\alpha}-\frac{1}{2}\ul{p}\Bigr)
\end{equation}
(cf.\ Gindikin~\cite{Gindikin64}).
Moreover,
we set
\[
A_r(\ul{\alpha})=\biggl(
\exp\Bigl\{\frac{\pi\sqrt{-1}}{2}\Bigl(
\sum_{j=1}^r\varepsilon_j\delta_j\alpha_j
+\frac{1}{2}\sum_{1\le j<k\le r}\varepsilon_j\delta_kn_{kj}
\Bigr)\Bigr\}
\biggr)_{\ue,\ud\in\Ir},
\]
where the index $\ue$ runs vertically and $\ud$ horizontally.
Then,
Proposition 4.3 of the previous paper \cite{N2020} gives
the following functional equation of local zeta functions:
\begin{equation}
\label{eq:FE}
\vlzf{\Fourier[f^*]}{\ul{s}}
=
\frac{\Gamma_\Omega(\ul{s}\sigma)}{(2\pi)^{|\ul{s}\sigma|}}
A_r\Bigl(\ul{s}\sigma-\frac{1}{2}\ul{p}\Bigr)\vdlzf{f^*}{\tau(\ul{s})}.
\end{equation}
Here,
we set $|\ul{\alpha}|:=\alpha_1+\cdots+\alpha_r$ for $\ul{\alpha}\in\C^r$.
Note that a functional equation between zeta functions $\bs{\zeta}(L;\,\ul{s})$, $\bs{\zeta}^*(L^*;\,\ul{s})$ of 
prehomogeneous vector space $(H,V)$ can be described as
\begin{equation}
\label{eq:FEzeta}
\bs{\zeta}^*\bigl(L^*;\,\tau(\ul{s})\bigr)
=
v(L)
\frac{\Gamma_{\Omega}\bigl(\ul{s}\sigma-\frac12(\ul{q}-\ul{p})\bigr)}{(2\pi)^{|\ul{s}\sigma|}}
A_r\Bigl(\ul{s}\sigma-\frac12\ul{q}\Bigr)
\bs{\zeta}\bigl(L;\,\ul{s})\quad(\ul{s}\in D),
\end{equation}
where $L$, $L^*$ are a lattice and its dual lattice, respectively,
and $D$ is a suitable complex domain
(see Theorem 4.7 in \cite{N2020} for more detail).
Since the case $r=1$ corresponds to the classical case of the Riemann zeta function,
we exclude this case so that we assume that $r\ge 2$ throughout this paper.

Set $\Digit=\{0,1\}^r$.
Let $\mathrm{sgn}$ be the sign function of $\R^{\times}=\R\setminus\{0\}$,
that is,
$\mathrm{sgn}(x)=x/|x|$, and we set $\mathrm{sgn}(0):=0$.
Let $\omega^{s,a}$ $(s\in\C,\ a=0,1)$ be the quasi-character of $\R^{\times}$
defined by $\omega^{s,a}(x):=\mathrm{sgn}(x)^a|x|^s$.
Using this, 
we introduce zeta distributions $\localZ{f}$ of $\Omega$ and $\dlocalZ{f^*}$ of $\Omega^*$,
parametrized by $\ul{a},\ul{b}\in\Digit$, 
by
\begin{equation}
\label{def:localZ}
\localZ{f}
:=
\int_V \prod_{j=1}^r\omega^{s_j,a_j}\bigl(\Delta_j(x)\bigr)\,f(x)d\mu(x)
\quad
(f\in\rapid,\ \ul{s}\in\C^r)
\end{equation}
and
\[
\dlocalZ{f^*}
:=
\int_V \prod_{j=1}^r\omega^{s_j,b_j}\bigl(\Delta^*_j(y)\bigr)\,f^*(y)d\mu^*(y)
\quad
(f^*\in\rapid,\ \ul{s}\in\C^r),
\]
where $d\mu^*$ is a suitable $H$ invariant measure associated with the dual prehomogeneous vector space.
We shall rewrite the functional equation~\eqref{eq:FE} by using zeta distributions in Section~\ref{sect:proof_main}.
We note that,
for symmetric cones viewed as homogeneous spaces of reductive groups,
zeta distributions are studied in \cite{BCK2018}.

Put
$\kvsub{\ue}:=\prod_{j=1}^r\varepsilon_j^{a_j}$ $(\ue\in\Ir,\ \ul{a}\in\mathbb{Z}^r)$.
The following lemma tells us that analytic properties of $\localZ{f}$ and $\dlocalZ{f^*}$ are the same as $\lzf{f}{\s}$ and $\dlzf{f^*}{\ul{s}}$, respectively.

\begin{lemma}
\label{lemma:transform}
{\rm (1)} For each $\ul{a}\in\Digit$,
one has
$\ds\localZ{f}=\sum_{\ue\in\Ir}
\kvsup{\ul{a}\sigma}
\lzf{f}{\ul{s}}.
$\\
{\rm (2)} For each $\ul{b}\in\Digit$,
one has
$\ds
\dlocalZ{f^*}=\sum_{\ud\in\Ir}\kvsub[\ul{b}\sigma_*]{\ud}\dlzf{f^*}{\ul{s}}.
$
%
\end{lemma}

\begin{proof}
We shall consider only the assertion (1) because we can prove the assertion (2) in the same line.
Set $\ce:=\varepsilon_1c_1+\cdots+\varepsilon_rc_r\in V$ where $c_j$ is a suitable basis of $V_{jj}$ $(j=1,\dots,r)$,
and let $\ul{e}_j\in\R^r$ be the row unit vector having one on the $j$-th position and zeros elsewhere.

Recalling the property~\eqref{eq:multiplier matrix} of the multiplier matrix $\sigma=(\sigma_{jk})_{1\le j,k\le r}$,
we have
\[
\Delta_j(\ce)=\varepsilon_1^{\sigma_{j1}}\cdots\varepsilon_r^{\sigma_{jr}}=\kvsup{\ul{e}_j\sigma}.
\]
Since $\Oe$ is the orbit of $H$ through $\ce\in V$
and since $\Delta_j$ is a relatively invariant function,
we have for $x\in\Oe$
\[
\omega^{s_j,a_j}\bigl(\Delta_j(x)\bigr)
=
\bigl(\varepsilon_1^{\sigma_{j1}}\cdots\varepsilon_r^{\sigma_{jr}}\bigr)^{a_j}
|\Delta_j(x)|^{s_j}
=\kvsup{a_j\,\ul{e}_j\sigma}|\Delta_j(x)|^{s_j},
\]
and thus
\[
\prod_{j=1}^r\omega^{s_j,a_j}\bigl(\Delta_j(x)\bigr)
=
\kvsup{\ul{a}\sigma}\,|\Delta_1(x)|^{s_1}\cdots|\Delta_r(x)|^{s_r}.
\]
By~\eqref{eq:lzf} and~\eqref{def:localZ}, 
we obtain the assertion (1).
\end{proof}

Let $J^{(r)}$ be an Hadamard matrix defined as an orthogonal symmetric matrix of size $2^r$ by
\[
J^{(1)}:=\frac{1}{\sqrt{2}}\pmat{1&1\\1&-1},\quad
J^{(r)}:=J^{(r-1)}\otimes J^{(1)}
=
\frac{1}{\sqrt{2}}\pmat{J^{(r-1)}&J^{(r-1)}\\ J^{(r-1)}&-J^{(r-1)}}.
\]
If we introduce a column vector $\kv$ $(\ul{a}\in\Digit)$ by
$\kv:=\bigl(\kvsup{\ul{a}}\bigr)_{\ue\in\Ir}$,
then we have $\set{\kv}{\ul{a}\in\Digit}=\mathcal{I}_2^{\otimes r}$
so that there exists a suitable ordering of $\Digit$ such that
a matrix $2^{-r/2}\bigl(\kv\bigr)_{\ul{a}\in\Digit}$ of size $2^r$ obtained by 
arraying column vectors $2^{-r/2}\kv$ in a row
satisfies 
\[
J^{(r)}=2^{-r/2}\bigl(\kv\bigr)_{\ul{a}\in\Digit}.
\]
We fix this order of $\Digit$ throughout the paper.
Then, we write zeta distributions in a vector form as 
\[
\bs{Z}(f;\,\ul{s}):=\bigl(\localZ{f}\bigr)_{\ul{a}\in\Digit},\quad
\bs{Z}^*(f^*;\,\ul{s}):=\bigl(\dlocalZ{f^*}\bigr)_{\ul{b}\in\Digit}.
\]

%
%
%
%

Since $\sigma$ is a unimodular matrix,
the map
$\Digit\ni\ul{a}\mapsto\ul{a}\sigma\ (\mathrm{mod}\ 2)\in\Digit$
is a bijection and hence
it induces a permutation in $\Digit$.
Thus, the matrix $J_\sigma:=2^{-r/2}\bigl(\kvsub[\ul{a}\sigma]{\ue}\bigr)_{\ul{a}\in\Digit}$
is described by using a permutation matrix $W_\sigma$ as $J_\sigma=J^{(r)}W_\sigma$.
Then, the equations in Lemma~\ref{lemma:transform} are written in a vector form as 
\[
\bs{Z}(f;\,\ul{s})={}^tJ_\sigma\vlzf{f}{\ul{s}}={}^tW_\sigma J^{(r)}\vlzf{f}{\ul{s}}.
\]
Similarly, 
with respect to $\Omega^*$,
we have a permutation matrix $W_{\sigma_*}$ such that
$J_{\sigma_*}=J^{(r)}W_{\sigma_*}$.
Therefore, we obtain a relationship between local zeta functions and zeta distributions in a vector form as
\begin{equation}
\label{eq:relation}
J^{(r)}\vlzf{f}{\ul{s}}=W_\sigma\bs{Z}(f;\,\ul{s}),\quad
J^{(r)}\vdlzf{f^*}{\ul{s}}=W_{\sigma_*}\bs{Z}^*(f^*;\,\ul{s}).
\end{equation}

%
%
%

\begin{remark}
In general,
matrices $W_\sigma$ and $W_{\sigma_*}$ induce different permutations.
For example, the Vinberg cone of rank 3, which is the non-symmetric homogeneous cone of the least dimension,
has structure constants $(n_{21},n_{31},n_{32})=(1,1,0)$.
The corresponding multiplier matrices $\sigma$ and $\sigma_*$ are given respectively as
\[
\sigma=\pmat{1&0&0\\1&1&0\\1&0&1},\quad
\sigma_*=\pmat{1&1&1\\0&1&0\\0&0&1},
\]
which make $W_\sigma$ and $W_{\sigma_*}$ different permutations.
\end{remark}

\section{Decomposition of gamma matrices}
\label{sect:proof_main}

All notations are followed by the previous section.
Note that we now assume that $r\ge 2$.
Set $\ul{\alpha}=\ul{s}\sigma-\frac{1}{2}\ul{p}$.
Then, the gamma matrix, that is
the coefficient matrix in \eqref{eq:FE} can be written, by~\eqref{eq:defofgamma}, as
\[
\frac{\Gamma_\Omega(\ul{s}\sigma)}{(2\pi)^{|\ul{s}\sigma|}}
A_r\Bigl(\ul{s}\sigma-\frac{1}{2}\ul{p}\Bigr)
=
\frac{\Gamma(\ul{s}\sigma-\frac{1}{2}\ul{p})}{(2\pi)^{|\ul{s}\sigma-\frac{1}{2}\ul{p}|}}
A_r\Bigl(\ul{s}\sigma-\frac{1}{2}\ul{p}\Bigr)
=
\frac{\Gamma(\ul{\alpha})}{(2\pi)^{|\ul{\alpha}|}}A_r(\ul{\alpha}).
\]
In the first equality,
we use $|\ul{p}|=n-r$.
According to Lemma~\ref{lemma:transform} together with the above observation,
in order to rewrite the functional equation~\eqref{eq:FEzeta} by using zeta distributions,
it is enough to consider the conjugate of $A_r(\ul{\alpha})$ with respect to $J^{(r)}$.
This means that
we can give a proof apart from the structure of homogeneous cones,
and hence
we replace half integers $\frac{n_{kj}}{2}$ by real numbers $\theta_{kj}$.
Let us put
\[
\ctt(\alpha):=\cos\frac{\pi \alpha}{2},\quad
\stt(\alpha):=\sin\frac{\pi\alpha}{2}\quad
(\alpha\in\C).
\]
By setting $\Theta_r:=(\theta_{kj})_{1\le j<k\le r}$ and $\dIr=\set{\ue\in\Ir}{\varepsilon_1=1}$,
we put
\[
C_r(\ul{\alpha};\,\Theta_r)=\left(\ctt\Bigl(\sum_{i=1}^r\varepsilon_i\delta_i\alpha_i+\sum_{1\le j<k\le r}\varepsilon_j\delta_k\theta_{kj}\Bigr)\right)_{\ue,\ud\in\dIr}
\]
and
\[
S_r(\ul{\alpha};\,\Theta_r)=\left(\stt\Bigl(\sum_{i=1}^r\varepsilon_i\delta_i\alpha_i+\sum_{1\le j<k\le r}\varepsilon_j\delta_k\theta_{kj}\Bigr)\right)_{\ue,\ud\in\dIr}.
\]
If we set $\ul{e}_1=(1,0,\dots,0)$,
then we have $S_r(\ul{\alpha};\,\Theta_r)=C_r(\ul{\alpha}-\ul{e}_1;\,\Theta_r)$
because $\varepsilon_1\delta_1=1$ and $\stt(a)=\ctt(a-1)$.
Since $J^{(r)}$ can be decomposed into 
\[
J^{(r)}=J^{(r-1)}\otimes J^{(1)}
=
\frac{1}{\sqrt{2}}
\pmat{J^{(r-1)}&0\\0&J^{(r-1)}}
\pmat{I_{2^{r-1}}&I_{2^{r-1}}\\I_{2^{r-1}}&-I_{2^{r-1}}},
\]
we have by the definition of ordering of $\Ir$
\begin{equation}
\label{eq:CS}
J^{(r)}
A_r(\ul{\alpha})
J^{(r)}
=
2
\smat{J^{(r-1)}C_r(\ul{\alpha};\,\Theta_r)J^{(r-1)}&0\\0&\sqrt{-1}\,J^{(r-1)}S_r(\ul{\alpha};\,\Theta_r)J^{(r-1)}}.
\end{equation}
Hence,
it is enough to consider the conjugate of $C_r(\ul{\alpha};\,\Theta_r)$ by $J^{(r-1)}$
instead of that of $A_r(\ul{\alpha})$ by $J^{(r)}$.
In order to state the formula,
we introduce two series $D_j^{(r)}(\alpha)$ $(j=1,\dots,r)$ and $P^{(r)}_{kj}(\theta)$ $(1\le j<k\le r)$ of matrices of size $2^{r-1}$ inductively as follows.
$D_j^{(r)}(\alpha)$  are diagonal matrices defined as
\[
\begin{array}{c}
\ds
D_1^{(2)}(\alpha)=\pmat{\ctt(\alpha)&0\\0&\stt(\alpha)},
\quad
D_2^{(2)}(\alpha)=\pmat{\ctt(\alpha)&0\\0&-\stt(\alpha)}\\
\ds
D_1^{(r)}(\alpha):=\pmat{\ctt(\alpha)I_{2^{r-2}}&0\\0&\stt(\alpha)I_{2^{r-2}}},
\quad
D_2^{(r)}(\alpha):=\pmat{D_1^{(r-1)}(\alpha)&0\\0&D_1^{(r-1)}(\alpha+1)},\\
\ds
D_k^{(r)}(\alpha):=D_{k-1}^{(r-1)}(\alpha)\otimes I_2
=
\pmat{D_{k-1}^{(r-1)}(\alpha)&0\\0&D_{k-1}^{(r-1)}(\alpha)}
\quad
(k=3,\dots,r).
\end{array}
\]
%
Setting $R_2=\pmat{0&-1\\1&0}$ and $A_2=\pmat{0&1\\1&0}$,
we put 
\[
X_{kj}^{(r)}=
\underbrace{
	\overbrace{I_2\otimes\cdots\otimes I_2}^{r-k-j+1}
	\otimes
	\overbrace{A_2\otimes\cdots\otimes A_2}^{k-2}
	\otimes R_2 \otimes
	\overbrace{I_2\otimes\cdots\otimes I_2}^{j-1}
}_{r-1}.
\]
Then, $P_{kj}^{(r)}(\theta)$ are defined as orthogonal matrices by
\begin{equation}\label{def:Pkj}
P_{kj}^{(r)}(\theta)=\exp\bigl(\theta X_{kj}^{(r)}\bigr)\quad(\theta\in\R).
\end{equation}
Note that, for any fixed $j$ and $r$, the matrices $\widetilde{P}_{kj}^{(r)}(\theta)$ $(k>j)$ are commutative.
For non-commutative variables $a_1,\dots,a_r$, we set
\[
\revProd a_j:=a_r\times a_{r-1}\times \cdots\times a_1.
\]

\begin{proposition}
\label{prop:conjugate C}
The conjugate of $C_r(\ul{\alpha};\,\Theta_r)$ by $J^{(r-1)}$ is given as
\begin{equation}
\label{formula:proposition}
J^{(r-1)}C_r(\ul{\alpha};\,\Theta_r)J^{(r-1)}
=
2^{r-1}
\revProd[j=1]\Bigl(\prod_{k=j+1}^r P^{(r)}_{kj}(\theta_{kj})\cdot D^{(r)}_j(\alpha_j)\Bigr).
\end{equation}
\end{proposition}

\noindent\textit{Proof.}
We shall prove this proposition by induction on rank $r$.
First, if $r=2$,
then a simple calculation yields that
\[
J^{(1)}C_2(\ul{\alpha};\,\Theta_2)J^{(1)}
=
2
\pmat{\ctt(\alpha_2)&0\\0&-\stt(\alpha_2)}
P_{21}^{(2)}(\theta_{21})
\pmat{\ctt(\alpha_1)&0\\0&\stt(\alpha_1)}
\]
and hence the theorem holds in the case $r=2$.
Suppose that $r\ge 3$ and the theorem holds for the case $r-1$.
Let us decompose $C_r(\ul{\alpha};\,\Theta_r)$ into four matrices $C_{11},C_{12},C_{21},C_{22}$ of size $2^{r-2}$ as
\[
C_r(\ul{\alpha};\,\Theta_r)=\pmat{C_{11}&C_{12}\\C_{21}&C_{22}}.
\]
For $i,j\le 2^{r-2}$,
suppose that $i$-th entry of $\Ir$ is $(1,1,\ue')$ and $j$-th $(1,1,\ud')$.
Then, since we now take the order of $\Ir$  as in Section~\ref{sect:preliminary},
the positions of the $(i,j)$ entries of matrices $C_{kl}$ $(k,l\in\{1,2\})$ 
as entries of $C_r(\ul{\alpha};\,\Theta_r)$ are given as
\[
\begin{array}{ll}
C_{11}\colon \ue_{11}(i)=(1,1,\ue'),&
\ud_{11}(j)=(1,1,\ud'),\\
C_{21}\colon \ue_{21}(i)=(1,-1,-\ue'),&
\ud_{21}(j)=(1,1,\ud'),\\[1ex]
C_{12}\colon \ue_{12}(i)=(1,1,\ue'),&
\ud_{12}(j)=(1,-1,-\ud'),\\[1ex]
C_{22}\colon \ue_{22}(i)=(1,-1,-\ue'),&
\ud_{22}(j)=(1,-1,-\ud').
\end{array}
\]
Therefore,
each $(i,j)$ entry of $C_{kl}$ is described as
\[
\begin{array}{ll}
C_{11}:&\ds\ctt\Bigl(\alpha_1+\alpha_2+\theta_{21}+\sum_{j=3}^r\varepsilon_j\delta_j\alpha_j+\sum_{k=3}^r\delta_k(\theta_{k1}+\theta_{k2})+\sum_{3\le j<k\le r}\varepsilon_j\delta_k\theta_{kj}\Bigr)\\[4ex]
C_{21}:&\ds\ctt\Bigl(\alpha_1-\alpha_2+\theta_{21}-\sum_{j=3}^r\varepsilon_j\delta_j\alpha_j+\sum_{k=3}^r\delta_k(\theta_{k1}-\theta_{k2})-\sum_{3\le j<k\le r}\varepsilon_j\delta_k\theta_{kj}\Bigr)\\
&\ds\quad=\ctt\Bigl(-(\alpha_1-\alpha_2+\theta_{21})+\sum_{j=3}^r\varepsilon_j\delta_j\alpha_j+\sum_{k=3}^r\delta_k(-\theta_{k1}+\theta_{k2})+\sum_{3\le j<k\le r}\varepsilon_j\delta_k\theta_{kj}\Bigr)\\[4ex]
C_{12}:&\ds\ctt\Bigl(\alpha_1-\alpha_2-\theta_{21}-\sum_{j=3}^r\varepsilon_j\delta_j\alpha_j-\sum_{k=3}^r\delta_k(\theta_{k1}+\theta_{k2})-\sum_{3\le j<k\le r}\varepsilon_j\delta_k\theta_{kj}\Bigr)\\
&\ds\quad=\ctt\Bigl(-(\alpha_1-\alpha_2-\theta_{21})+\sum_{j=3}^r\varepsilon_j\delta_j\alpha_j+\sum_{k=3}^r\delta_k(\theta_{k1}+\theta_{k2})+\sum_{3\le j<k\le r}\varepsilon_j\delta_k\theta_{kj}\Bigr)\\[4ex]
C_{22}:&\ds\ctt\Bigl(\alpha_1+\alpha_2-\theta_{21}+\sum_{j=3}^r\varepsilon_j\delta_j\alpha_j+\sum_{k=3}^r\delta_k(-\theta_{k1}+\theta_{k2})+\sum_{3\le j<k\le r}\varepsilon_j\delta_k\theta_{kj}\Bigr).
\end{array}
\]
Let us set
\[
\alpha'(a,b):=(a(\alpha_1+a\alpha_2+b\theta_{21}),\alpha_3,\dots,\alpha_r)\in\C^{r-1}
\quad(a,b=+\text{ or }-)
\]
and
\[
\Theta'(\pm):=(\theta'_{kj})_{1\le j<k\le r-1},\quad
\begin{cases}
\theta_{k1}'=\pm\theta_{k+1,1}+\theta_{k+1,2}&(k\ge 2),\\[1ex]
\theta'_{kj}=\theta_{k+1,j+1}&(k>j\ge 2).
\end{cases}
\]
Then, we obtain
\[
\begin{array}{cc}
C_{11}=C_{r-1}(\alpha'(+,+);\,\Theta'(+)),&
C_{12}=C_{r-1}(\alpha'(-,-);\,\Theta'(+)),\\[1ex]
C_{21}=C_{r-1}(\alpha'(-,+);\,\Theta'(-)),&
C_{22}=C_{r-1}(\alpha'(+,-);\,\Theta'(-)).
\end{array}
\]
Put $J=J^{(r-2)}$.
Then, the hypothesis of induction implies that
the conjugate of $C_{kl}$ $(k,l\in\{1,2\})$ with respect to $J$ are given as
\[
\begin{array}{l}
\ds
JC_{11}J=
2^{r-2}\cdot T\cdot\prod_{k=3}^{r}P^{(r-1)}_{k-1,1}(\theta_{k1}+\theta_{k2})\times 
D^{(r-1)}_1(\alpha_1+\alpha_2+\theta_{21})\\[4ex]
\ds
JC_{21}J=
2^{r-2}\cdot T\cdot\prod_{k=3}^{r}P^{(r-1)}_{k-1,1}(-\theta_{k1}+\theta_{k2})\times 
D^{(r-1)}_1(-(\alpha_1-\alpha_2+\theta_{21}))\\[4ex]
\ds
JC_{12}J=
2^{r-2}\cdot T\cdot\prod_{k=3}^{r}P^{(r-1)}_{k-1,1}(\theta_{k1}+\theta_{k2})\times 
D^{(r-1)}_1(-(\alpha_1-\alpha_2-\theta_{21}))\\[4ex]
\ds
JC_{22}J=
2^{r-2}\cdot T\cdot\prod_{k=3}^{r}P^{(r-1)}_{k-1,1}(-\theta_{k1}+\theta_{k2})\times 
D^{(r-1)}_1(\alpha_1+\alpha_2-\theta_{21}),
\end{array}
\]
where $T$ is a matrix of size $2^{r-2}$ which is a product of common factors among them given as
\[
T=
D^{(r-1)}_{r-1}(\alpha_r)
P^{(r-1)}_{r-1,r-2}(\theta'_{r-2,r-1})
\cdots
D^{(r-1)}_3(\alpha_4)
\left(
\prod_{k=3}^r
P^{(r-1)}_{k2}(\theta'_{k2})
\right)
D^{(r-1)}_2(\alpha_3).
\]
If we set $\widetilde{T}=T\otimes I_2$, then we see by definition that 
$\widetilde{T}$ is exactly a product of elements before $D_{3}^{(r)}(\alpha_3)$ in \eqref{formula:proposition}.
Moreover, it is easily verified that it is commutative with $\widetilde{K}=I\otimes J^{(1)}=\frac{1}{\sqrt{2}}\smat{I&I\\I&-I}$,
where we set $I=I_{2^{r-2}}$.

For brevity,
we write $D_{ij}$ for the $D_1^{(r-1)}$ part of $JC_{ij}J$ $(i,j=1,2)$,
and $P_{\pm}$ for the product of $P_{kj}^{(r-1)}$'s where
$\pm$ corresponds to the signature before $\theta_{k1}$.
Then, 
we have
\[
\begin{array}{r@{\ }c@{\ }l}
J^{(r-1)}C_r(\ul{\alpha};\,\Theta_r)J^{(r-1)}
&=&
\ds
\frac{1}{2}\pmat{J&J\\J&-J}
\pmat{C_{11}&C_{12}\\C_{21}&C_{22}}
\pmat{J&J\\J&-J}\\[2em]
&=&\ds
2^{r-2}\cdot 
\frac12\pmat{I&I\\I&-I}
\pmat{TP_+D_{11}&TP_+D_{12}\\ TP_-D_{12}&TP_-D_{22}}
\pmat{I&I\\I&-I}\\[2em]
&=&
\ds
2^{r-2}\cdot 
\widetilde{K}
\pmat{T&0\\0&T}
\pmat{P_+&0\\0&P_-}
\pmat{D_{11}&D_{12}\\D_{21}&D_{22}}
\widetilde{K}\\[2em]
&=&
\ds
2^{r-2}\cdot 
\widetilde{T}\cdot \widetilde{K}
\pmat{P_+&0\\0&P_-}
\widetilde{K}\cdot
\widetilde{K}
\pmat{D_{11}&D_{12}\\D_{21}&D_{22}}
\widetilde{K}.
\end{array}
\] 

\begin{lemma}
\label{lemma:1}
One has
\[
\widetilde{K}
\pmat{D_{11}&D_{12}\\D_{21}&D_{22}}
\widetilde{K}
=
2D^{(r)}_2(\alpha_2)P^{(r)}_{21}(\theta_{21})D^{(r)}_1(\alpha_1).
\]
\end{lemma}
\begin{proof}
Let us denote by $I'=I_{2^{r-3}}$ the identity matrix of size $2^{r-3}$.
By definition of $D_{ij}$ $(i,j=1,2)$, we have
\[
\begin{array}{l}
\pmat{D_{11}&D_{12}\\D_{21}&D_{22}}\\[1.5em]
\quad
=
\smat{\ctt(\alpha_1+\alpha_2+\theta_{21})I'&0&\ctt(-(\alpha_1-\alpha_2-\theta_{21}))I'&0\\
0&\stt(\alpha_1+\alpha_2+\theta_{21})I'&0&\stt(-(\alpha_1-\alpha_2-\theta_{21}))I'\\
\ctt(-(\alpha_1-\alpha_2+\theta_{21}))I'&0&\ctt(\alpha_1+\alpha_2-\theta_{21})I'&0\\
0&\stt(-(\alpha_1-\alpha_2+\theta_{21}))I'&0&\stt(\alpha_1+\alpha_2-\theta_{21})I'}\\[3em]
\quad
=
\smat{\ctt(\alpha_1+\alpha_2+\theta_{21})I'&0&\ctt(\alpha_1-\alpha_2-\theta_{21})I'&0\\
0&\ctt(\alpha_1+(\alpha_2-1)+\theta_{21})I'&0&\ctt(\alpha_1-(\alpha_2-1)-\theta_{21})I'\\
\ctt(\alpha_1-\alpha_2+\theta_{21})I'&0&\ctt(\alpha_1+\alpha_2-\theta_{21})I'&0\\
0&\ctt(\alpha_1-(\alpha_2-1)+\theta_{21})I'&0&\ctt(\alpha_1+(\alpha_2-1)-\theta_{21})I'}
\end{array}
\]
so that we can apply the result of the case $r=2$.
Since
\[
\begin{array}{l}
J^{(1)}\pmat{\ctt(\alpha_1+\alpha_2+\theta_{21})&\ctt(\alpha_1-\alpha_2-\theta_{21})\\
\ctt(\alpha_1-\alpha_2+\theta_{21})&\ctt(\alpha_1+\alpha_2-\theta_{21})}J^{(1)}\\
\ds\qquad=
2
\pmat{\ctt(\alpha_2)&0\\0&-\stt(\alpha_2)}
P^{(2)}_{21}(\theta_{21})
\pmat{\ctt(\alpha_1)&0\\0&\stt(\alpha_1)}
\end{array}
\]
and
\[
\begin{array}{l}
\ds J^{(1)}\pmat{\ctt(\alpha_1+(\alpha_2-1)+\theta_{21})&\ctt(\alpha_1-(\alpha_2-1)-\theta_{21})\\
\ctt(\alpha_1-(\alpha_2-1)+\theta_{21})&\ctt(\alpha_1+(\alpha_2-1)-\theta_{21})}J^{(1)}\\[1.5em]
\ds\qquad=
2
\pmat{\ctt(\alpha_2-1)&0\\0&-\stt(\alpha_2-1)}
P^{(2)}_{21}(\theta_{21})
\pmat{\ctt(\alpha_1)&0\\0&\stt(\alpha_1)},
\end{array}
\]
the diagonal matrix with respect to $\alpha_2$ can be described as
\[
\mathrm{diag}
\bigl(\ctt(\alpha_2)I',\ctt(\alpha_2-1)I',-\stt(\alpha_2)I',-\stt(\alpha_2-1)I'\bigr)
=
\mathrm{diag}
\bigl(\ctt(\alpha_2)I',\stt(\alpha_2)I',\ctt(\alpha_2+1)I',\stt(\alpha_2+1)I'\bigr),
\]
where we use $\ctt(a-1)=\stt(a)$ in this equation.
Therefore, we obtain
\[
\begin{array}{r@{\ }c@{\ }l}
\ds
\widetilde{K}
\pmat{D_{11}&D_{12}\\D_{21}&D_{22}}
\widetilde{K}
&=&\ds
2\pmat{D^{(r-1)}_1(\alpha_2)&0\\0&D^{(r-1)}_1(\alpha_2+1)}
\pmat{\ctt(\theta_{21})I&-\stt(\theta_{21})I\\ \stt(\theta_{21})I&\ctt(\theta_{21})I}
\pmat{\ctt(\alpha_1)I&0\\0&\stt(\alpha_1)I}\\[1.5em]
&=&\ds
2D^{(r)}_2(\alpha_2)P^{(r)}_{21}(\theta_{21})D^{(r)}_1(\alpha_1),
\end{array}
\]
as required.
\end{proof}

\begin{lemma}
\label{lemma:2}
One has
\[
\widetilde{K}
\pmat{P_+&0\\0&P_-}
\widetilde{K}
=
\prod_{k=3}^{r}
P^{(r)}_{k2}(\theta_{k2})
\cdot
\prod_{k=3}^{r}
\widetilde{K}
\pmat{P_{k-1,1}^{(r-1)}(\theta_{k1})&0\\0&{}^{\,t\!}P_{k-1,1}^{(r-1)}(\theta_{k1})}
\widetilde{K}.
\]
\end{lemma}	
\begin{proof}
By definition of $P_{kj}^{(r-1)}$,
we see that
$P^{(r-1)}_{kj}(\theta+\phi)=P^{(r-1)}_{kj}(\theta)P^{(r-1)}_{kj}(\phi)$ and
$P^{(r-1)}_{kj}(-\theta)={}^{\,t\!}P^{(r-1)}_{kj}(\theta)$.
The hypothesis of induction tells us that
each $P^{(r-1)}_{k-1,1}$ $(k=3,\dots,r)$ are mutually commutative and hence
we have 
\[
\begin{array}{r@{\ }c@{\ }l}
\pmat{P_+&0\\0&P_-}
&=&\ds
\pmat{\ds\prod_{k=3}^{r}P^{(r-1)}_{k-1,1}(\theta_{k1}+\theta_{k2})&0\\0&\ds\prod_{k=3}^{r}P^{(r-1)}_{k-1,1}(-\theta_{k1}+\theta_{k2})}\\[3.5em]
&=&
\ds
\prod_{k=3}^{r}
\pmat{P_{k-1,1}^{(r-1)}(\theta_{k2})&0\\0&P_{k-1,1}^{(r-1)}(\theta_{k2})}
\prod_{k=3}^{r}
\pmat{P_{k-1,1}^{(r-1)}(\theta_{k1})&0\\0&P_{k-1,1}^{(r-1)}(-\theta_{k1})}\\[2.5em]
&=&
\ds
\prod_{k=3}^{r}
P^{(r)}_{k2}(\theta_{k2})
\prod_{k=3}^{r}
\pmat{P_{k-1,1}^{(r-1)}(\theta_{k1})&0\\0&{}^{\,t\!}P_{k-1,1}^{(r-1)}(\theta_{k1})}.
\end{array}\]
Since 
$P^{(r)}_{k2}(\theta_{k2})$ commutes with $\widetilde{K}$ for any $k\ge 3$,
we obtain the lemma.
\end{proof}

\begin{lemma}
\label{lemma:3}
For $k\ge 3$, one has
\[
\widetilde{K}
\pmat{P_{k-1,1}^{(r-1)}(\theta_{k1})&0\\0&{}^{\,t\!}P_{k-1,1}^{(r-1)}(\theta_{k1})}
\widetilde{K}
\cdot
D^{(r)}_2(\alpha_2)
=
D^{(r)}_2(\alpha_2)
P^{(r)}_{k1}(\theta_{k1}).
\]
\end{lemma}
\begin{proof}
Recall the definition of $P_{kj}^{(r)}$ in \eqref{def:Pkj}.
Let $Y'_{k-1}$ be a matrix of size $2^{r-3}$ defined by $X_{k-1,1}^{(r-1)}=Y'_{k-1}\otimes R_2$.
Then, $Y'_{k-1}$ is a symmetric matrix satisfying $(Y')^2=I'$,
and we have
\[
P^{(r-1)}_{k-1,1}(\theta_{k1})=\pmat{\ctt(\theta_{k1})I'&-\stt(\theta_{k1})Y'_{k-1}\\ \stt(\theta_{k1})Y'_{k-1}&\ctt(\theta_{k1})I'}.
\]
Therefore, we obtain
\[
\begin{array}{l}
\ds
\widetilde{K}
\pmat{P_{k-1,1}^{(r-1)}(\theta_{k1})&0\\0&{}^{\,t\!}P_{k-1,1}^{(r-1)}(\theta_{k1})}
\widetilde{K}\\[1em]
\qquad=
\ds\frac{1}{2}
\pmat{
P_{k-1,1}^{(r-1)}(\theta_{k1})+{}^{\,t\!}P_{k-1,1}^{(r-1)}(\theta_{k1})
&
P_{k-1,1}^{(r-1)}(\theta_{k1})-{}^{\,t\!}P_{k-1,1}^{(r-1)}(\theta_{k1})\\
P_{k-1,1}^{(r-1)}(\theta_{k1})-{}^{\,t\!}P_{k-1,1}^{(r-1)}(\theta_{k1})
&
P_{k-1,1}^{(r-1)}(\theta_{k1})+{}^{\,t\!}P_{k-1,1}^{(r-1)}(\theta_{k1})
}\\[1em]
\qquad=
\ds
\smat{\ctt(\theta_{k1})I'&0&0&-\stt(\theta_{k1})Y'_{k-1}\\
0&\ctt(\theta_{k1})I'&\stt(\theta_{k1})Y'_{k-1}&0\\
0&-\stt(\theta_{k1})Y'_{k-1}&\ctt(\theta_{k1})I'&0\\
\stt(\theta_{k1})Y'_{k-1}&0&0&\ctt(\theta_{k1})I'
}.
\end{array}
\]
On the other hand,
by definition of $P^{(r)}_{k1}(\theta_{k1})$ as in \eqref{def:Pkj},
we have
\[
\begin{array}{r@{\ }c@{\ }l}
P_{k1}^{(r)}(\theta_{k1})
&=&
\exp\left(\theta_{k1}Y'_{k-1}\otimes A_2\otimes R_2\right)\\
&=&
\ds
\smat{
\ctt(\theta_{k1})I'&0&0&-\stt(\theta_{k1})Y'_{k-1}\\
0&\ctt(\theta_{k1})I'&-\stt(\theta_{k1})Y'_{k-1}&0\\
0&\stt(\theta_{k1})Y'_{k-1}&\ctt(\theta_{k1})I'&0\\
\stt(\theta_{k1})Y'_{k-1}&0&0&\ctt(\theta_{k1})I'
}.
\end{array}
\]
Since $D_2^{(r)}(\alpha_2)$ is a diagonal matrix of the form
\[
D_2^{(r)}(\alpha_2)=\mathrm{diag}(\ctt(\alpha_2)I',\,\stt(\alpha_2)I',\,-\stt(\alpha_2)I',\,\ctt(\alpha_2)I'),
\]
now we can easily verify the assertion.
\end{proof}

We now return to the proof of Proposition~\ref{prop:conjugate C}.
By Lemmas~\ref{lemma:1}, \ref{lemma:2} and \ref{lemma:3},
we finally obtain
\[
\begin{array}{l}
J^{(r-1)}C_r(\ul{\alpha};\,\Theta_r)J^{(r-1)}\\
\qquad=
\ds
2^{r-2}\cdot
\widetilde{T}\cdot \widetilde{K}
\pmat{P_+&0\\0&P_-}
\widetilde{K}
\cdot 
\widetilde{K}
\pmat{D_{11}&D_{12}\\D_{21}&D_{22}}
\widetilde{K}\\[1em]
\qquad=\ds
2^{r-2}\cdot
\widetilde{T}\cdot
\prod_{k=3}^{r}
P^{(r)}_{k2}(\theta_{k2})
\cdot
\prod_{k=3}^{r}
\widetilde{K}
\pmat{P_{k-1,1}^{(r-1)}(\theta_{k1})&0\\0&{}^{\,t\!}P_{k-1,1}^{(r-1)}(\theta_{k1})}
\widetilde{K}\\[1em]
\qquad\quad\ds\quad\times
2D^{(r)}_2(\alpha_2)P^{(r)}_{21}(\theta_{21})D^{(r)}_1(\alpha_1)\\[1em]
\qquad=
\ds
2^{r-1}
\widetilde{T}
\prod_{k=3}^{r}
P^{(r)}_{k2}(\theta_{k2})
\cdot D^{(r)}_2(\alpha_2)
\prod_{k=3}^r
P^{(r)}_{k1}(\theta_{k1})
\cdot
P^{(r)}_{21}(\theta_{21})
D^{(r)}_1(\alpha_1),
\end{array}
\]
which is exactly the form of \eqref{formula:proposition},
and hence we have proved Proposition~\ref{prop:conjugate C}.
\qed

\vspace{1ex}

By Proposition~\ref{prop:conjugate C}, we have obtained
\[
\begin{array}{r@{\ }c@{\ }l}
J^{(r)}A_r(\ul{\alpha})J^{(r)}
&=&
\ds
2
\pmat{J^{(r-1)}C_r(\ul{\alpha};\,\Theta_r)J^{(r-1)}&0\\
0&\sqrt{-1}\,J^{(r-1)}C_r(\ul{\alpha}-\ul{e}_1;\,\Theta_r)J^{(r-1)}}\\
&=&
\ds
2^r\pmat{I&0\\0&\sqrt{-1}\,I}
\revProd[j=2]
\left\{
\prod_{k=j+1}^r
\pmat{P^{(r)}_{kj}(\theta_{kj})&0\\0&P^{(r)}_{kj}(\theta_{kj})}
\cdot
\pmat{D^{(r)}_j(\alpha_j)&0\\0&D^{(r)}_j(\alpha_j)}
\right\}\\
& &
\ds
\quad\times
\pmat{P^{(r)}_{k1}(\theta_{k1})&0\\0&P^{(r)}_{k1}(\theta_{k1})}
\cdot
\pmat{D^{(r)}_1(\alpha_1)&0\\0&D^{(r)}_1(\alpha_1-1)}
\end{array}
\]
By assigning gamma factors $\frac{\Gamma(\alpha_i)}{(2\pi)^{\alpha_i}}$ to 
each diagonal entry,
we obtain the following theorem.

\begin{theorem}
\label{theorem}
For each $j=1,\dots,r$ and $1\le j<k\le r$,
one sets
\[
\begin{array}{l}
\ds
\widetilde{D}^{(r)}_1(\alpha_1)
=
\frac{\Gamma(\alpha_1)}{(2\pi)^{\alpha_1}}
\pmat{D^{(r)}_1(\alpha_1)&0\\0&D^{(r)}_1(\alpha_1-1)},\\
\ds
\widetilde{D}^{(r)}_j(\alpha_j)
=
\frac{\Gamma(\alpha_j)}{(2\pi)^{\alpha_j}}
\pmat{D^{(r)}_j(\alpha_j)&0\\0&D^{(r)}_j(\alpha_j)}\quad(j=1,\dots,r),\\
\ds
\widetilde{P}^{(r)}_{kj}(\theta_{kj})
=
\pmat{P^{(r)}_{kj}(\theta_{kj})&0\\0&P^{(r)}_{kj}(\theta_{kj})}
\quad(1\le j<k\le r).
\end{array}
\]
Then the zeta distributions $\bs{Z}(f;\,\ul{s})$ and $\bs{Z}^*(f^*;\,\ul{s})$ satisfies the following functional equation
\[
W_\sigma\bs{Z}(\mathcal{F}\relax[f^*];\,\ul{s})
=
\mathcal{A}_r\Bigl(\ul{s}\sigma-\frac12\ul{p}\Bigr)
W_{\sigma_*}\bs{Z}^*(f^*;\,\tau(\ul{s})),
\]
where $\mathcal{A}_r(\ul{\alpha})$ $(\ul{\alpha}\in\C^r)$ is a product of variable-wise matrices given as
\[
\mathcal{A}_r(\ul{\alpha})
=
2^r
\pmat{I_{2^{r-1}}&0\\0&\sqrt{-1}\,I_{2^{r-1}}}
\revProd[j=1]\Bigl(
\prod_{k=j+1}^r \widetilde{P}^{(r)}_{kj}(\theta_{kj})
\cdot
\widetilde{D}^{(r)}_j(\alpha_j)\Bigr).
\]
\end{theorem}


Proposition~\ref{prop:conjugate C} enables us to compute the determinant of gamma matrices.
Set
\[
\widetilde{A}_r(\ul{\alpha},\,\Theta_r)=\frac{\Gamma(\ul{\alpha})}{(2\pi)^{|\ul{\alpha}|}}A_r(\ul{\alpha},\,\Theta_r)\quad(\ul{\alpha}\in\C^r).
\]
Note that $\widetilde{A}_r(\ul{s}\sigma-\frac12\ul{p},\,\Theta_r)$
for $\Theta_{r}=\bigl(\frac{n_{kj}}{2}\bigr)_{1\le j<k\le r}$ is exactly the gamma matrix in \eqref{eq:FE}.

\begin{corollary}
For $r\ge 2$, one has
\[
\det \widetilde{A}_r(\ul{\alpha},\Theta_r)
=
\left(\frac{\Gamma(\ul{\alpha})}{(2\pi)^{|\ul{\alpha}|}}\right)^{2^r}
\Bigl(\prod_{j=1}^r
2\sin\pi\alpha_j\Bigr)^{2^{r-1}}.
\]
\end{corollary}
\begin{proof}
Let $\varepsilon(r)=\bigl(\sqrt{-1}\bigr)^{2^{r-1}}=(-1)^{2^{r-2}}$.
By \eqref{eq:CS} and Proposition~\ref{prop:conjugate C}, 
we have
\[
\begin{array}{r@{\ }c@{\ }l}
\ds
\det \widetilde{A}_r(\ul{\alpha},\,\Theta_r)
&=&
\ds
\left(\frac{\Gamma(\ul{\alpha})}{(2\pi)^{|\ul{\alpha}|}}\right)^{2^r}
\cdot
2^{2^r}
\det C_r(\ul{\alpha};\,\Theta_r)
\cdot
\bigl(\sqrt{-1}\bigr)^{2^{r-1}}
\det S_r(\ul{\alpha};\,\Theta_r)\\
&=&
\ds
\varepsilon(r)
\left(\frac{\Gamma(\ul{\alpha})}{(2\pi)^{|\ul{\alpha}|}}\right)^{2^r}
\cdot
2^{2^r}
\det C_r(\ul{\alpha};\,\Theta_r)
\cdot
\det C_r(\ul{\alpha}-\ul{e}_1;\,\Theta_r)\\
&=&
\ds
\varepsilon(r)
\left(\frac{\Gamma(\ul{\alpha})}{(2\pi)^{|\ul{\alpha}|}}\right)^{2^r}
\cdot 2^{r\cdot2^r}
\det D_1^{(r)}(\alpha_1)
\det D_1^{(r)}(\alpha_1-1)
\prod_{j=2}^r
\bigl(\det D^{(r)}_j(\alpha_j)\bigr)^2.
\end{array}
\]
Here,
we use facts that $P^{(r)}_{kj}(\theta_{kj})$ are all orthogonal matrices
and $S_r(\ul{\alpha};\,\Theta_r)=C_r(\ul{\alpha}-\ul{e}_1;\,\Theta_r)$.
By definition of $D^{(r)}_1(\alpha_1)$,
we readily see that
\[
\det D_1^{(r)}(\alpha_1)
=
\bigl(\ctt(\alpha_1)\stt(\alpha_1)\bigr)^{2^{r-2}},\quad
\det D_1^{(r)}(\alpha_1-1)
=
\varepsilon(r)\bigl(\ctt(\alpha_1)\stt(\alpha_1)\bigr)^{2^{r-2}}.
\]
Moreover,
we have
\begin{equation}
\label{eq:coro}
\bigl(\det D^{(r)}_j(\alpha_j)\bigr)^2=\bigl(\ctt(\alpha_j)\stt(\alpha_j)\bigr)^{2^{r-1}}
\quad(j=2,\dots,r).
\end{equation}
In fact,
in the case $j=2$,
we have $\det D_2^{(2)}(\alpha_2)=-\ctt(\alpha_2)\stt(\alpha_2)$
by definition, 
and if $r\ge 3$, then
\[
\begin{array}{r@{\ }c@{\ }l}
\det D_2^{(r)}(\alpha_2)
&=&
\det D_1^{(r-1)}(\alpha_2)\det D_1^{(r-1)}(\alpha_2+1)\\
&=&
(\ctt(\alpha_2)\stt(\alpha_2))^{2^{r-3}}
\cdot
(-\ctt(\alpha_2)\stt(\alpha_2))^{2^{r-3}}\\
&=&
(-1)^{2^{r-3}}
(\ctt(\alpha_2)\stt(\alpha_2))^{2^{r-2}}.
\end{array}
\]
In the cases $j\ge 3$, 
if we set $\delta(r)=-1$ if $r=2,3$, and $\delta(r)=1$ otherwise,
then
\[
\begin{array}{r@{\ }c@{\ }l}
\ds
\det D^{(r)}_j(\alpha_j)
&=&
(\det D^{(r-1)}_{j-1}(\alpha_j))^2
=
(\det D^{(r-2)}_{j-2}(\alpha_j))^{2^2}
=
\cdots
\\
&=&
\ds
(\det D^{(r-j+2)}_{2}(\alpha_j))^{2^{j-2}}
=
\Bigl\{\delta(r)
\bigl(\ctt(\alpha_j)\stt(\alpha_j)\bigr)^{2^{(r-j+2)-2}}\Bigr\}^{2^{j-2}}\\
&=&
\ds
\bigl(\ctt(\alpha_j)\stt(\alpha_j)\bigr)^{2^{r-2}},
\end{array}
\]
and hence we have verified \eqref{eq:coro}.
Therefore,
since $\ctt(\alpha)\stt(\alpha)=\frac{\sin\pi\alpha}{2}$,
we see that
\[
\det \widetilde{A}_r(\ul{\alpha};\,\Theta_r)
=
\left(\frac{\Gamma(\ul{\alpha})}{(2\pi)^{|\ul{\alpha}|}}\right)^{2^r}
\cdot 2^{r\cdot2^r}
\prod_{j=1}^r
\left(\frac{\sin\pi\alpha_j}{2}\right)^{2^{r-1}}
=
\left(\frac{\Gamma(\ul{\alpha})}{(2\pi)^{|\ul{\alpha}|}}\right)^{2^r}
\Bigl(\prod_{j=1}^r
2\sin\pi\alpha_j\Bigr)^{2^{r-1}},
\]
which shows the corollary.
\end{proof}

\section{Completion}
\label{sect:completion}

We keep all notations used in the previous sections.
In this section,
we shall consider a completion of local zeta functions $\vlzf{f}{\ul{s}}$ and $\vdlzf{f^*}{\ul{s}}$.
To do so,
let us consider the following condition
\begin{equation}
\label{condition}
\text{for fixed $m=0,1$, one has }\frac{\pi}{4}\sum_{j<k}\varepsilon_j\delta_kn_{kj}\equiv m\pi\ (\textrm{mod }2\pi)
\text{ for all }\ue,\ud\in\Ir.
\end{equation}
This condition means that we can ignore the sum part of $n_{kj}$.
More precisely, 
we would consider the case $\theta_{kj}=0$ for all $1\le j<k\le r$ so that
\[
A_r(\ul{\alpha})=(-1)^m\Bigl(\exp\Bigl\{\frac{\pi\sqrt{-1}}{2}\sum_{j=1}^r\varepsilon_j\delta_j\alpha_j\Bigr\}\Bigr)_{\ue,\ud\in\Ir},
\]
which implies that $P_{kj}^{(r)}(\theta_{kj})$ are the identity matrix for all $1\le j<k\le r$.
Hence the main theorem tells us that 
the conjugate of the gamma matrix with respect to $J^{(r)}$ is diagonal.
Set
\begin{equation}
\label{eq:defofD}
\Lambda_r(\ul{\alpha}):=\mathrm{diag}\left(\frac{\pi^{|\ul{\alpha}|/2}}{\Gamma\bigl(2^{-1}(\ul{\alpha}+\ul{a})\bigr)}\right)_{\ul{a}\in\Digit}.
\end{equation}
Then, we have the following proposition.

\begin{proposition}
\label{prop:completion}
Suppose the condition~\eqref{condition}.
If one sets
\[
\bs{\Psi}(f;\,\ul{s}):=\Lambda_r(\ul{s}\sigma-2^{-1}\ul{p})
J^{(r)}\vlzf{f}{\ul{s}},\quad
\bs{\Psi}^*(f^*;\,\ul{s}):=\Lambda_r(\ul{s}\sigma_*-2^{-1}\ul{q})
J^{(r)}\vdlzf{f^*}{\ul{s}},
\]
then one has for $\ul{s}\in\C^r$ and $f^*\in\rapid$
\[
\bs{\Psi}(\mathcal{F}\relax[f^*];\,\ul{s})=\efactor \bs{\Psi}^*(f^*;\,\tau(\ul{s})),
\quad
\text{where}\quad
\efactor=(-1)^m\,\mathrm{diag}\Bigl((\sqrt{-1})^{|\ul{a}|}\Bigr)_{\ul{a}\in\Digit}.
\]
\end{proposition}

\begin{proof}
We already know that,
under the assumption~\eqref{condition}, 
the matrix $A_r(\ul{\alpha})$ can be diagonalized by $J^{(r)}$
so that each $\kv$ is an eigenvector.
Therefore,
in order to determine $\varepsilon$-factor,
we calculate eigenvalues of $A_r(\ul{\alpha})$, which correspond to $\kv$ for all $\ul{a}\in\Digit$.
Let us fix an $\ul{a}\in\Digit$.
Then, an entry of $A_r(\ul{\alpha})\,\kv$ corresponding to the parameter $\ue\in\Ir$ is described as
\[
(-1)^m
\sum_{\ud\in\Ir}\exp\Bigl(\frac{\pi\sqrt{-1}}{2}\sum_{j=1}^r\delta_j\cdot \varepsilon_j\alpha_j\Bigr)
\cdot \kvsub{\ud}
=
(-1)^m\cdot
2^r(\sqrt{-1})^{|\ul{a}|}\prod_{j=1}^r\ctt(\varepsilon_j\alpha_j-a_j).
\]
This equation can be easily obtained by induction on $r$.
Since
$\ctt(\varepsilon z-a)=\varepsilon^a\ctt(z-a)$ 
for $\varepsilon\in\{1,-1\}$ and $a\in\{0,1\}$,
we have
\[
\prod_{j=1}^r\ctt(\varepsilon_j\alpha_j-a_j)
=
\prod_{j=1}^r\ctt(\alpha_j-a_j)
\cdot \kvsub[\ul{a}]{\ue}.
\]
This implies that
the eigenvalue corresponding to $\kv$ is 
\[(-1)^m\cdot 2^r(\sqrt{-1})^{|\ul{a}|}\prod_{j=1}^r\ctt(\alpha_j-a_j).\]

By \eqref{eq:FE}, 
we next consider the following product
\[
2^r\prod_{j=1}^r\frac{\Gamma(\alpha_j)}{(2\pi)^{\alpha_j}}
\ctt(\alpha_j-a_j)\quad(\ul{a}\in\Digit).
\]
Let us recall
Euler's reflection formula and Legendre's duplication formula:
\[
\Gamma(z)\Gamma(1-z)=\frac{\pi}{\sin\pi z},\quad
\Gamma(z)=\frac{2^z}{2\sqrt{\pi}\ }\Gamma\Bigl(\frac{z}{2}\Bigr)
\Gamma\Bigl(\frac{z+1}{2}\Bigr)\quad(z\in \C).
\]
Combining these two formulas, we obtain the following formula
\[
\Gamma(z)\,\ctt(z-a)=
\frac{2^z\sqrt{\pi}}{2}\cdot
\frac{\Gamma\bigl(2^{-1}(z+a)\bigr)}{\Gamma\bigl(2^{-1}(1-z+a)\bigr)}
\quad(z\in\C;\ a=0\text{ or }1).
\]
Therefore,
we have for each $\ul{a}\in\Digit$
\[
\begin{array}{r@{\ }c@{\ }l}
\ds
2^r
\prod_{j=1}^r\frac{\Gamma(\alpha_j)}{(2\pi)^{\alpha_j}}\ctt(\alpha_j-a_j)
&=&
\ds
\frac{2^r}{(2\pi)^{|\ul{\alpha}|}}
\prod_{j=1}^r
\frac{2^{\alpha_j}\sqrt{\pi}}{2}\cdot
\frac{\Gamma\bigl(2^{-1}(\alpha_j+a_j)\bigr)}{\Gamma\bigl(2^{-1}(1-\alpha_j+a_j)\bigr)}\\
&=&
\ds
\pi^{r/2-|\ul{\alpha}|}
\cdot\frac{\Gamma(2^{-1}(\ul{\alpha}+\ul{a}))}{\Gamma\bigl(2^{-1}(\ul{1}-\ul{\alpha}+\ul{a})\bigr)}.
\end{array}
\]
Recall that when we consider $A_r(\ul{\alpha})$ as a gamma matrix of functional equation~\eqref{eq:FE}, 
the vector $\ul{\alpha}$ is defined to be  $\ul{\alpha}=\ul{s}\sigma-\frac12\ul{p}$.
Thus, we have
\[
\ul{1}-\ul{\alpha}
=
\ul{1}-\ul{s}\sigma+\frac{1}{2}\ul{p}
=
\ul{d}-\ul{s}\sigma-\frac{1}{2}\ul{q}
=
\tau(\ul{s})\sigma_*-\frac{1}{2}\ul{q},
\]
and
\[
\frac12(\bigl|\ul{1}-\ul{\alpha}\bigr|-|\ul{\alpha}|)
=
\frac12(\ul{1}-2|\ul{\alpha}|)
=
\frac{r}{2}-|\ul{\alpha}|
\]
by the facts that 
$\ul{d}=\ul{1}+(\ul{p}+\ul{q})/2$ and $\ul{d}-\ul{s}\sigma=\tau(\ul{s})\sigma_*$.
These observations imply that
\[
2^r\prod_{j=1}^r\frac{\Gamma(\alpha_j)}{(2\pi)^{\alpha_j}}
\ctt(\alpha_j-a_j)
=
\frac{\pi^{|\tau(\ul{s})\sigma_*-2^{-1}\ul{q}|/2}}{\Gamma\bigl(2^{-1}(\tau(\ul{s})\sigma_*-2^{-1}\ul{q}+\ul{a})\bigr)}
\cdot
\frac{\Gamma\bigl(2^{-1}(\ul{s}\sigma-2^{-1}\ul{p}+\ul{a})\bigr)}{\pi^{|\ul{s}\sigma-2^{-1}\ul{p}|/2}},
\]
whence the proposition is now proved.
\end{proof}

\begin{remark}
We note that there exist homogeneous cones satisfying~\eqref{condition}.
In fact, 
the exceptional symmetric cone $\mathrm{Herm}(3,\mathbb{O})^+$ obviously satisfies~\eqref{condition}.
Other examples are given as homogeneous cones $\Omega$ such that
$n_{kj}=0$ or $4$ for all $j<k$.
Such homogeneous cones can be constructed from chordal and $A_4$-free graphs 
(see Letac--Massam~\cite{LM2007} for definition) as follows.
Let $G$ be a chordal and $A_4$-free graph of size $n$.
Set 
\[V_G:=\set{x\in\mathrm{Herm}(n,\mathbb{H})}{x_{ij}=0\text{ if }i\not\sim j\text{ in }G}.
\]
Then, $\Omega_G:=V_G\cap\mathrm{Herm}(n,\mathbb{H})^+$ is a desired homogeneous cone.
\end{remark}


\begin{remark}
The condition~\eqref{condition} can be rewritten as the following condition%
\footnote{Professor Ochiai kindly informs the author on this equivalence.}:
\begin{center}
all $n_{kj}$ are even, and all entries of $\ul{p}$ and $\ul{q}$ are divisible by 4.
\end{center}
In fact, by using an obvious equation $\varepsilon_j=1-(1-\varepsilon_j)$, we have
\[
\sum_{j<k}\varepsilon_j\delta_kn_{kj}
=
\sum_{j<k}n_{kj}-\sum_{j=1}^r(1-\varepsilon_j)q_j-\sum_{k=1}^r(1-\delta_k)p_k+\sum_{j<k}(1-\varepsilon_j)(1-\delta_k)n_{kj}.
\]
Since the condition imposes that the residual divided by $8$ is independent of the choice of $\ue,\ud\in\Ir$,
we first see that $p_k$ and $q_j$ both need to be divisible by $4$,
and then, by considering the last term, we see that $n_{kj}$ must be even.
\end{remark}


\section*{Acknowledgments}

This work was supported by Grant-in-Aid for JSPS fellows (2018J00379).
The present author is grateful to Professor Hiroyuki Ochiai for insightful comments for this work.

\end{document}